\documentclass[11pt]{amsart}

\usepackage[T1]{fontenc}
\usepackage{amsmath,amssymb,amsthm}
\usepackage[margin=1.1in]{geometry}
\usepackage{xcolor}
\usepackage{hyperref}

\hypersetup{
  colorlinks=true,
  linkcolor=blue,
  citecolor=blue,
  urlcolor=blue
}

\newcommand{\N}{\mathbb{N}}
\newcommand{\Z}{\mathbb{Z}}

\newcommand{\F}{\mathbb{F}}
\newcommand{\cA}{\mathcal{A}}
\newcommand{\cT}{\mathcal{T}}
\newcommand{\boldz}{\mathbf{z}}
\DeclareMathOperator{\Jac}{Jac}

\theoremstyle{plain}
\newtheorem{thm}{Theorem}[section]
\newtheorem{lem}[thm]{Lemma}
\newtheorem{prop}[thm]{Proposition}

\newtheorem{introthm}{Theorem}

\theoremstyle{definition}
\newtheorem{defn}[thm]{Definition}
\newtheorem{notation}[thm]{Notation}

\theoremstyle{remark}

\numberwithin{equation}{section}

\title[Arithmetic complexity of factorial multiples]
{Square-root logarithmic bounds for the arithmetic complexity of factorial multiples}

\author{Liang Zhang}
\email{tensor@pku.edu.cn}
\thanks{Corresponding author: Liang Zhang.}

\subjclass[2020]{Primary 11Y55; Secondary 11N25, 11G20, 68Q15}
\keywords{arithmetic complexity, factorial multiples, smooth numbers, short intervals, genus-two Jacobians}
\date{}

\begin{document}

\begin{abstract}
Assume that the Riemann zeta function has no zeros in
$\Re(s)>1-\delta$, $s\neq1$, for some fixed $0<\delta\leq1/2$.
We prove that, for every sufficiently large integer $k$, a positive
multiple $m_k k!$ has arithmetic complexity
$\tau(m_k k!)\leq\exp(C\sqrt{\log k\log\log k})$, where $C$ depends
only on $\delta$. We first establish a smooth-number asymptotic at every
sufficiently large center $x$ in intervals of length at least $x^{3/4}$.
Combining this estimate with an established genus-two construction yields
the arithmetic bound without Cheng's conjecture on smooth numbers in
elliptic Hasse intervals. The bound includes the cost of constructing all
integer parameters. This work was carried out with assistance from the
UniMath model, developed by UniDT (Shanghai) Co., Ltd., and OpenAI's
GPT-6.1 Sol and GPT-6 Astra.
\end{abstract}

\maketitle

\section{Introduction}\label{sec:introduction}

The purpose of this paper is to bound the arithmetic complexity of a
suitable positive multiple of each factorial. The main result,
Theorem~\ref{thm:main}, obtains the scale
$\exp(O_\delta(\sqrt{\log k\log\log k}))$ from a fixed zero-free
half-plane for the Riemann zeta function, without assuming Cheng's
smooth-number conjecture for elliptic Hasse intervals.
Here $\tau(n)$ measures the number of additions, subtractions and
multiplications needed to construct $n$ from $1$ and $2$, and $\mu(k)$
denotes the least such cost among positive multiples of $k!$.
The precise circuit model and both quantities are specified in
Definition~\ref{defn:complexity}.

The arithmetic cost of integers was studied by de Melo and Svaiter
\cite{deMeloSvaiter1996}, who obtained elementary universal bounds and
lower bounds valid for almost all integers. Moreira \cite{Moreira1997}
sharpened the latter to the scale $\log n/\log\log n$ and established a
matching asymptotic upper scale for all sufficiently large integers.
These general estimates leave substantial room for special sequences:
repeated squaring produces some very large integers with very short
programs, and an almost-all lower bound does not determine the complexity
of factorials or of their easiest multiples. Shub and Smale
\cite{ShubSmale1995} showed that failure of a polynomial bound in $\log k$
for ultimate factorial complexity would imply
$P_{\mathbb C}\neq NP_{\mathbb C}$. Upper bounds for $\mu(k)$ thus form
part of a broader study of explicit integers and algebraic complexity.

Further connections arise from the arithmetic complexity of explicit
integer sequences. Koiran \cite{Koiran2004} relates the construction of
integers such as $\lfloor2^n\log2\rfloor$ to Valiant's theory of
arithmetic circuits. In particular, a polynomial-size constant-free
circuit family for the permanent would give a polynomial bound in
$\log n$ for the cost of this sequence. With the additional assumption
$P=PSPACE$, his results give such a bound for $\tau(n!)$ as well.
Thus lower bounds for particular integer sequences can have consequences
for polynomial circuit complexity, beyond the almost-all estimates for
integers of a given size.

There is also a computational line of work on concrete upper bounds.
Markstr\"om \cite{Markstrom2015} uses exhaustive searches for short
straight-line programs to determine exact values and bounds for small
factorials, primorials, and their ultimate complexities.
These calculations complement asymptotic constructions by exhibiting
short programs for individual integers. On the asymptotic side,
successive multiplication gives $\mu(k)\leq\tau(k!)=O(k)$, while
Strassen's factorial algorithm gives the bound
$\tau(k!)=O(\sqrt{k}(\log k)^2)$; see
\cite[Section~2]{Markstrom2015}. The possibility of choosing a multiple
of $k!$ permits the much smaller bounds discussed next.

Cheng \cite[Conjecture~1 and Theorem~1]{Cheng2004} obtained
\[
 \mu(k)\leq\exp\!\left(O(\sqrt{\log k\log\log k})\right)
\]
under a conjecture about smooth numbers in the elliptic Hasse intervals
$[p+1-2\sqrt p,p+1+2\sqrt p]$.
Writing $L_p[a]=\exp(a\sqrt{\log p\log\log p})$, that conjecture asks,
for every fixed $a>0$, for $\sqrt p\,L_p[-1/(2a)+o(1)]$ many
$L_p[a]$-smooth integers in such an interval.
His construction uses elliptic-curve division polynomials and also gives
a randomized subexponential-time procedure for producing the programs.

The interval length in this conjecture is essential. Younis
\cite[Theorem~1.3]{Younis2024}, assuming RH, proves smooth-number
asymptotics for $x^\theta\leq h\leq x$ and
$(\log x)^{K(\theta)}\leq y\leq2x$, for each fixed $\theta>1/2$.
Cheng's smoothness scale exceeds this lower bound on $y$, but his
intervals have length $4\sqrt p$ at centers comparable to $p$.
For every fixed $\theta>1/2$, $4\sqrt p/p^\theta\to0$.
Consequently, Younis's RH theorem does not establish Cheng's conjecture:
the endpoint interval scale remains outside its stated range.

A different algebraic construction changes the required interval scale.
Zhang \cite[Theorem~1.2]{Zhang2026} obtained the unconditional bound
\[
 \mu(k)\leq
 \exp\!\left(O((\log k)^{2/3}(\log\log k)^{1/3})\right)
\]
using genus-two Jacobians, a family with a non-torsion marked divisor
class, and integer circuits encoding nonzero arithmetic certificates.
The order-counting theorem of Lenstra, Pila and Pomerance
\cite[Theorem~1.1]{LPP02} supplies the relevant genus-two orders.
The intervals used in this construction are
\[
 I_p=\left[p^2-\frac12p^{3/2},\,p^2+\frac12p^{3/2}\right].
\]
Their centers are of size $p^2$ and their lengths are $p^{3/2}$,
corresponding to the exponent $3/4$ rather than $1/2$.

The main contribution here is to recover Cheng's square-root logarithmic
upper-bound scale by combining this genus-two construction with a
short-interval smooth-number estimate under a quasi-Riemann hypothesis.
More precisely, any fixed $\delta>0$ with
$\zeta(s)\neq0$ for $\Re(s)>1-\delta$, $s\neq1$, suffices.
Theorem~\ref{thm:strip} proves the required density estimate at every
sufficiently large center on the $x^{3/4}$ interval scale, with smoothness
cutoff $\exp(\sqrt{\log x\log\log x})$.
The transfer in Section~\ref{sec:arithmetic} then yields the bound for
$\mu(k)$ with all integer gates charged. This removes the need for
Cheng's smooth-number conjecture in obtaining that upper-bound scale;
it does not prove the conjecture itself. The genus-two family and
certificate construction are those of \cite{Zhang2026}, while the
analytic argument adapts Younis's contour method to the fixed-strip
hypothesis.

Our main result is the following.

\begin{introthm}[Positive factorial multiples]\label{thm:main}
Assume that a fixed real number $0<\delta\leq1/2$ satisfies
\[
 \zeta(s)\neq0\qquad\bigl(\Re(s)>1-\delta,\ s\neq1\bigr).
\]
There are constants $C>0$ and $k_0\geq3$, depending only on $\delta$, such that, for every integer $k\geq k_0$, there is a positive integer $m_k$ satisfying
\[
 \tau(m_k k!)
 \leq
 \exp\!\left(C\sqrt{\log k\log\log k}\right).
\]
All logarithms are natural. The constants may depend on $\delta$ but are independent of $k$.
\end{introthm}

The multiplier is a positive integer existence witness. The circuit computes the multiple $m_k k!$ from $1$ and $2$.

The new ingredient concerns smooth integers at every sufficiently large real center. We use the counting function $\Psi(x,y)$ and saddle point $\alpha(x,y)$ defined in Definition~\ref{defn:smooth}.

\begin{introthm}[Smooth integers under a fixed-strip hypothesis]\label{thm:strip}
Assume that a fixed real number $0<\delta\leq 1/2$ satisfies
\[
 \zeta(s)\neq 0
 \qquad
 \bigl(\Re(s)>1-\delta,\ s\neq 1\bigr).
\]
For $x>e$, put
\[
 L=\log x,\qquad
 S=\sqrt{L\log L},\qquad
 y=e^S,\qquad
 u=\frac{L}{S},
\]
and let $\alpha=\alpha(x,y)>0$ be determined by
\[
 \sum_{\substack{q\leq y\\q\ {\rm prime}}}
 \frac{\log q}{q^\alpha-1}=L.
\]
There are constants $C_\delta$ and $X_\delta$, depending only on $\delta$, such that, for every real $x\geq X_\delta$ and every real $h$ with
$x^{3/4}\leq h\leq x$,
\[
 \left|
 \frac{\Psi(x+h,y)-\Psi(x,y)}
      {\alpha(h/x)\Psi(x,y)}
 -1
 \right|
 \leq \frac{C_\delta}{u}.
\]
After increasing $X_\delta$, the same range satisfies
\[
 \Psi(x+h,y)-\Psi(x,y)
 \geq
 \frac{h}{4}
 \exp\!\left(-3\sqrt{\log x\log\log x}\right).
\]
\end{introthm}

The analytic method underlying Theorem~\ref{thm:strip} follows
\cite[Section~3]{Younis2024}. In particular, the construction of the contour
as the right envelope of a central path, paths associated with zeta zeros,
and a remaining vertical line is adapted from that work.
Sections~\ref{sec:preliminaries} and~\ref{sec:analytic} supply the details
needed to apply this method under the fixed hypothesis
$\zeta(s)\neq0$ for $\Re(s)>1-\delta$, $s\neq1$.
Section~\ref{sec:preliminaries} collects the unconditional estimates used
in the argument, and Section~\ref{sec:analytic} verifies the local shift,
the envelope construction, and the contributions of all contour pieces for
$y=\exp(\sqrt{\log x\log\log x})$ and $x^{3/4}\leq h\leq x$.
These verifications yield thresholds depending on $\delta$ and estimates
uniform in $h$; the underlying contour construction is due to Younis.

The factor $\alpha$ is retained throughout, including at $h=x$.
Recent AI-assisted work addresses the fixed-half-plane hypothesis directly
in the classical setting. OpenAI's preprints \cite{OpenAI2026SevenEighths,OpenAI2026ElevenTwelfths}
report proofs of nonvanishing for $\Re(s)>7/8$ and $\Re(s)>11/12$,
respectively; the latter is a shorter argument written with human assistance.
These conclusions supply the hypothesis of Theorems~\ref{thm:main}
and~\ref{thm:strip} with $\delta=1/8$ and $\delta=1/12$, respectively.
Cureton and Thorner \cite[Theorem~1.1]{CuretonThorner2026v1} report the
corresponding result for Dedekind zeta functions over arbitrary number fields,
with $\delta=1/6000$. For our application, the classical Riemann zeta
function alone is sufficient. We state both theorems with their common
zero-free hypothesis explicit and prove the resulting implications below;
either of the reported classical nonvanishing results then yields an
unconditional specialization.
In that specialization, the deduction of the square-root logarithmic
bound is complete using the cited arithmetic inputs and
Theorem~\ref{thm:strip}; no separate conjecture about smooth numbers in
elliptic Hasse intervals is required.

The local shift uses \cite[Lemma~2.3]{Younis2024}, with the fixed positive strip width making its errors small at the chosen value of $y$. Zero density bounds the sum of the retained zero paths. The other analytic inputs are the unconditional estimates in \cite[(1.3), (1.5), (1.6), (1.9), (1.13), and Lemma~3.6]{Younis2024}.

The arithmetic construction is inherited from \cite{Zhang2026}. Its marked family is given by \cite[(7)--(9) and Lemma~4.1]{Zhang2026}. We use its charged positive certificate \cite[Proposition~5.1]{Zhang2026} and its transfer to the marked slice \cite[Lemma~6.1]{Zhang2026}. Recording rational failures and amplifying prime support are the constructions of \cite[Lemmas~3.2 and~3.1]{Zhang2026}. These are established ingredients of the present proof.

Section~\ref{sec:preliminaries} fixes the analytic notation and records the exact unconditional inputs. Section~\ref{sec:analytic} proves Theorem~\ref{thm:strip}, including the local shift, central curves and horizontal connectors. Section~\ref{sec:arithmetic} gives the density-to-complexity transfer with all gate costs and proves Theorem~\ref{thm:main}.

\section{Preliminaries}\label{sec:preliminaries}

We use the smooth-number notation and saddle-point estimates of \cite{Younis2024}, with natural logarithms and real interval endpoints.

The von Mangoldt function is
\[
 \Lambda(n)=
 \begin{cases}
  \log q,&n=q^j,\quad q\text{ prime},\quad j\geq1,\\
  0,&\text{otherwise}.
 \end{cases}
\]

\begin{defn}[Smooth integers and the saddle point]\label{defn:smooth}
A positive integer is $y$-smooth if all its prime divisors are at most $y$; the integer $1$ is included. For real $x\geq y\geq 2$, define
\[
 \Psi(x,y)=
 \#\{n\in\N\mid n\leq x,\ n\text{ is }y\text{-smooth}\},
 \qquad
 Z_y(s)=\prod_{\substack{q\leq y\\q\ {\rm prime}}}
          (1-q^{-s})^{-1}
 \quad(\Re(s)>0).
\]
Let $\alpha=\alpha(x,y)>0$ be the unique solution of
\[
 F_y(\alpha)=\log x,
 \qquad
 F_y(a)=\sum_{\substack{q\leq y\\q\ {\rm prime}}}
              \frac{\log q}{q^a-1}.
\]
Equivalently, $\alpha$ uniquely minimizes $x^a Z_y(a)$ over $a>0$.
\end{defn}

The function $F_y$ is continuous and strictly decreasing on $(0,\infty)$. Its limits at the two endpoints are $\infty$ and $0$, respectively. This proves existence and uniqueness in Definition~\ref{defn:smooth}. The equivalence follows from
\[
 \frac{d}{da}\log(x^aZ_y(a))=\log x-F_y(a).
\]

For $\Re(s)>0$, each Euler factor is holomorphic and nonzero, since $|q^{-s}|<1$. Differentiating the finite product and expanding a geometric series gives
\begin{equation}\label{eq:logarithmic-derivative}
 -\frac{Z_y'(s)}{Z_y(s)}
 =\sum_{q\leq y}\frac{\log q}{q^s-1}
 =\sum_{q\leq y}\sum_{j\geq1}(\log q)q^{-js}.
\end{equation}
On the other hand,
\[
 \sum_{n\leq y}\Lambda(n)n^{-s}
 =\sum_{\substack{q\ {\rm prime},\ j\geq1\\q^j\leq y}}
   (\log q)q^{-js}.
\]
Consequently,
\begin{equation}\label{eq:prime-power-remainder}
 -\frac{Z_y'(s)}{Z_y(s)}
 =\sum_{n\leq y}\Lambda(n)n^{-s}+R_y(s),
 \qquad
 R_y(s):=\sum_{\substack{q\leq y,\ j\geq2\\q^j>y}}
                 (\log q)q^{-js},
\end{equation}
where $q$ ranges over primes. Only exponents $j\geq2$ occur in the remainder, because $q\leq y$.

\begin{notation}\label{notation:analytic}
We write
\[
 u=\frac{\log x}{\log y},
 \qquad W=x^\alpha Z_y(\alpha).
\]
\end{notation}

\subsection{Unconditional saddle-point inputs}

The following statements are used in their first-version forms. Each is independent of an RH assumption.

\begin{prop}[Saddle-point mass, {\cite[(1.3)]{Younis2024}}]\label{prop:mass}
For real $2\leq y\leq x$,
\[
 W\ll \Psi(x,y)\sqrt{\log x\log y}.
\]
\end{prop}

\begin{prop}[Location of the saddle point, {\cite[(1.5)]{Younis2024}}]\label{prop:location}
For $\log x<y\leq x$ in the sufficiently large range,
\[
 \alpha
 =
 1-\frac{\log(u\log(u+1))+O(1)}{\log y},
\]
with a uniform bounded error.
\end{prop}

\begin{prop}[Decay away from the real axis, {\cite[(1.6)]{Younis2024}}]\label{prop:decay}
Fix $r>0$. For $\log x<y\leq x$ sufficiently large in terms of $r$, and
\[
 \frac{1}{\log y}\leq |t|
 \leq \exp\!\left((\log y)^{3/2-r}\right),
\]
there is a constant $c_r>0$ such that
\[
 \frac{|Z_y(\alpha+it)|}{Z_y(\alpha)}
 \ll_r
 \exp\!\left(
 -\frac{c_r u t^2}{t^2+(1-\alpha)^2}
 \right).
\]
\end{prop}

\begin{prop}[Truncated counting integral, {\cite[(1.9)]{Younis2024}}]\label{prop:perron}
For $2\leq y,h,H\leq x$,
\[
 \Psi(x+h,y)-\Psi(x,y)
 =
 \frac{1}{2\pi i}
 \int_{\alpha-iH}^{\alpha+iH}
 Z_y(s)\frac{(x+h)^s-x^s}{s}\,ds
 +O\!\left(\frac{x\log x}{H}\right).
\]
\end{prop}

\begin{prop}[The central integral, {\cite[Lemma~3.6]{Younis2024}}]\label{prop:central}
Uniformly for $2\leq y\leq x$,
\[
 \frac{1}{2\pi i}
 \int_{\alpha-i/\log y}^{\alpha+i/\log y}
 x^sZ_y(s)\,ds
 =
 \alpha\Psi(x,y)
 \left(1+O\!\left(\frac1u+\frac{\log y}{y}\right)\right).
\]
On the same segment,
\[
 \int_{\alpha-i/\log y}^{\alpha+i/\log y}
 |x^sZ_y(s)|\,|ds|
 \ll \Psi(x,y)
\]
and
\[
 \int_{\alpha-i/\log y}^{\alpha+i/\log y}
 |s-1|\,|x^sZ_y(s)|\,|ds|
 \ll
 \Psi(x,y)\frac{\log(u+1)}{\log y}.
\]
\end{prop}

These are the saddle-point and central-integral estimates recorded in the cited source. In particular, its smoothness restriction (1.2) is not imposed on the inputs above in the ranges just stated.

\subsection{A local zero-free estimate}

\begin{prop}[Local von Mangoldt sum, {\cite[Lemma~2.3]{Younis2024}}]\label{prop:local-input}
Let $z\geq2$, $0<e<1/10$ and $0<\eta\leq1/2$. Assume that $\zeta$ has no zeros in
\[
 \Re(s)>1-\eta,\qquad
 |\Im(s)-t|<10z^\eta.
\]
For $s=a+it\neq1$ with
$1-(1-e)\eta\leq a<1+1/\log z$,
\begin{align}
 \left|\sum_{n\leq z}\Lambda(n)n^{-s}\right|
 &\leq
 \frac{z^{1-a}+1}{|1-s|}
 +O_e\!\left((\log(|t|+2))^b\right)
 +O\!\left(\frac{(\log z)^2}{z^{e\eta}}\right)
 \notag\\
 &\quad+
 O_e\!\left(
 \frac{
 ((\log(|t|+z^\eta+2))^b+1/\eta)
 \log(z^\eta/\eta)}
 {z^{e\eta/2}}
 \right),\label{eq:local-input}
\end{align}
where $b=b(e)<1$ and the constants are uniform in the remaining parameters.
\end{prop}

This local estimate assumes neither RH nor the smoothness restriction (1.2) of the cited source.

\subsection{Zero density}

\begin{prop}[Zero density, {\cite[(1.13)]{Younis2024}}]\label{prop:zero-density}
Let $N(\sigma,T)$ count nontrivial zeros $\rho=\beta+i\gamma$ of $\zeta$ with $\beta\geq\sigma$ and $|\gamma|\leq T$, with multiplicity. For each fixed $r>0$,
\[
 N(\sigma,T)\ll_r T^{(30+r)(1-\sigma)/13}
 \qquad
 \left(\frac12\leq\sigma\leq1,\ T\geq2\right).
\]
Consequently, there is an absolute constant $C\geq1$ such that
\[
 N(1-v,T)\leq C T^{(12/5)v}
 \qquad
 \left(0\leq v\leq\frac12,\ T\geq2\right).
\]
\end{prop}

For the last statement, take $r=1/100$ and use
$(30+1/100)/13<12/5$. This is the Guth--Maynard zero-density estimate in the form quoted by \cite[(1.13)]{Younis2024}.

\section{Fixed-strip short intervals}\label{sec:analytic}

In this section, we prove Theorem~\ref{thm:strip} by establishing the local shift and then estimating every part of the counting contour.

\subsection{A long-interval estimate}

\begin{prop}[Long-interval density and a saddle-point lower bound]\label{prop:elementary}
For all real $x\geq y\geq16$,
\begin{equation}\label{eq:elementary-density}
 \Psi(x,y)
 \geq
 x\exp\!\left(
 -\log y-\frac{\log x}{\log y}\log(8\log x)
 \right).
\end{equation}
There is an absolute $x_0>e$ such that, for every real $x\geq x_0$, writing
\[
 L=\log x,\qquad S=\sqrt{L\log L},\qquad y=e^S,
\]
we have
\begin{equation}\label{eq:long-density}
 \Psi(x,y)\geq xe^{-3S}
\end{equation}
and
\begin{equation}\label{eq:alpha-lower}
 \alpha(x,y)\geq1-\frac{2\log L}{S}>\frac12.
\end{equation}
\end{prop}

\begin{proof}
We first obtain the prime-counting estimate needed for the argument. Let $\pi(t)$ count primes at most $t$, and take $t\geq16$. Put $n=\lfloor t/2\rfloor$. Then $n\geq8$ and $n\geq t/4$. Since the central binomial coefficient is the largest coefficient of order $2n$,
\[
 \binom{2n}{n}\geq\frac{4^n}{2n+1}\geq2^n.
\]
Here $2n+1\leq2^n$ for $n\geq3$, as follows by induction. For each prime $q$, let $v_q(N)$ denote the exponent of $q$ in the prime factorization of $N$. Then
\[
 v_q\!\binom{2n}{n}
 =
 \sum_{j\geq1}
 \left(
 \left\lfloor\frac{2n}{q^j}\right\rfloor
 -
 2\left\lfloor\frac{n}{q^j}\right\rfloor
 \right).
\]
Each summand is either zero or one, and it vanishes when $q^j>2n$. Thus
$q^{v_q(\binom{2n}{n})}\leq2n$ for every prime $q\leq2n$, and
\[
 2^n\leq\binom{2n}{n}
 \leq(2n)^{\pi(2n)}
 \leq t^{\pi(t)}.
\]
Taking logarithms gives
\begin{equation}\label{eq:prime-lower}
 \pi(t)\geq\frac{n\log2}{\log t}
 \geq\frac{t\log2}{4\log t}
 \geq\frac{t}{8\log t}.
\end{equation}

Now let $x\geq y\geq16$, and put
\[
 L=\log x,\qquad v=\log y,\qquad
 r=\left\lfloor\frac{L}{v}\right\rfloor,\qquad N=\pi(y).
\]
Unique factorization distinguishes the products associated with multisets of $r$ primes at most $y$. Every such product is at most $y^r\leq x$. Therefore
\[
 \Psi(x,y)\geq\binom{N+r-1}{r}
 =\frac{N(N+1)\cdots(N+r-1)}{r!}
 \geq\left(\frac Nr\right)^r
 \geq\left(\frac{y}{8vr}\right)^r.
\]
The inequalities $L-v\leq rv\leq L$ and $r\leq L/v$ imply
\[
 \log\Psi(x,y)
 \geq rv-r\log(8vr)
 \geq L-v-\frac Lv\log(8L).
\]
This proves \eqref{eq:elementary-density}.

For $S=\sqrt{L\log L}$ and $y=e^S$, eventually $16\leq y\leq x$ and $L\geq8$. Hence
\[
 \log\frac{\Psi(x,y)}x
 \geq-S-\frac LS\log(8L)
 \geq-S-\frac{2L\log L}{S}
 =-3S.
\]

It remains to bound $\alpha$. Set
\[
 a=1-\frac{2\log L}{S},
\]
which exceeds $1/2$ for large $L$. By \eqref{eq:prime-lower} and
$\pi(\sqrt y)\leq\sqrt y$,
\[
 \pi(y)-\pi(\sqrt y)
 \geq\frac{y}{8S}-\sqrt y
 \geq\frac{y}{16S}
\]
for sufficiently large $L$. Indeed, $e^{S/2}/(16S)\to\infty$. For
$\sqrt y<q\leq y$, we have $\log q>S/2$ and $q^a-1\leq y^a$. Thus
\[
 F_y(a)
 \geq\frac{y}{16S}\frac{S}{2y^a}
 =\frac{y^{1-a}}{32}
 =\frac{L^2}{32}>L
\]
when $L>32$. Since $F_y$ is strictly decreasing and $F_y(\alpha)=L$, this gives $\alpha>a$. One absolute threshold handles all the preceding conditions.
\end{proof}

\subsection{A local horizontal shift}

\begin{lem}[Fixed-width local shift]\label{lem:local-shift}
Fix real constants $0<d\leq1/2$ and $0<e<1/10$. There is $X_0(d,e)$ with the following property. Let
\[
 x\geq X_0(d,e),\qquad
 \exp((\log\log x)^2)\leq y\leq x,\qquad
 d\leq\eta\leq\frac12.
\]
Let $\sigma_1<\sigma_2$ satisfy
\[
 1-(1-e)\eta\leq\sigma_1<\sigma_2<1+\frac1{\log y},
\]
and let $0<|t|\leq x$. Assume that $\zeta$ has no zeros in
\[
 \Re(s)>1-\eta,\qquad
 \Im(s)\in(t-10y^\eta,t+10y^\eta).
\]
If
\[
 |t|\geq\frac4{e\log x}\bigl(y^{1-\sigma_1}+1\bigr),
\]
then
\[
 |x^{\sigma_1+it}Z_y(\sigma_1+it)|
 \leq
 x^{-(1-e)(\sigma_2-\sigma_1)}
 |x^{\sigma_2+it}Z_y(\sigma_2+it)|.
\]
\end{lem}

\begin{proof}
Put $L=\log x$ and let $a\in[\sigma_1,\sigma_2]$. Apply Proposition~\ref{prop:local-input} with $z=y$. Since $t\neq0$, its first term is at most
\[
 \frac{y^{1-\sigma_1}+1}{|t|}\leq\frac{eL}{4}.
\]
The logarithms involving $t$ in the error terms are $O(L)$, because
$|t|,y\leq x$ and $\eta\leq1/2$. Also,
\[
 (\log L)^2\leq\log y\leq L,\qquad \eta\geq d.
\]
The last two errors in \eqref{eq:local-input} are consequently bounded by
\[
 O_{d,e}\!\left(L^2e^{-ed(\log L)^2}\right)
 \quad\text{and}\quad
 O_{d,e}\!\left((L^b+1)L e^{-(ed/2)(\log L)^2}\right),
\]
respectively. Each is $o(L)$ uniformly: after division by $L$, the negative quadratic term in $\log L$ dominates the logarithms of the remaining factors. The other error is $O_e(L^b)=o(L)$. Enlarging $X_0(d,e)$ makes their sum at most $eL/4$. Thus
\[
 \left|\sum_{n\leq y}\Lambda(n)n^{-a-it}\right|
 \leq\frac{eL}{2}
 \qquad(\sigma_1\leq a\leq\sigma_2).
\]

By \eqref{eq:prime-power-remainder}, it remains to control $R_y(a+it)$. Since $a\geq1-(1-e)\eta\geq(1+e)/2$, the geometric-series identity gives
\begin{align*}
 |R_y(a+it)|
 &\leq\sum_{q\leq y}(\log q)\sum_{j\geq2}q^{-ja}
 =\sum_{q\leq y}\frac{(\log q)q^{-2a}}{1-q^{-a}}\\
 &\leq\frac1{1-2^{-(1+e)/2}}
        \sum_{q\leq y}\frac{\log q}{q^{1+e}}\\
 &\leq\frac1{1-2^{-(1+e)/2}}
        \sum_{n\geq2}\frac{\log n}{n^{1+e}}
 =O_e(1).
\end{align*}
This convergent series depends only on $e$. Increasing $X_0(d,e)$ once more gives
\[
 \left|\frac{Z_y'(a+it)}{Z_y(a+it)}\right|\leq eL
 \qquad(\sigma_1\leq a\leq\sigma_2).
\]
Finally,
\[
 \frac{d}{da}\log|Z_y(a+it)|
 =
 \Re\frac{Z_y'(a+it)}{Z_y(a+it)}.
\]
Integration from $\sigma_2$ to $\sigma_1$ yields
\[
 \log\frac{|Z_y(\sigma_1+it)|}{|Z_y(\sigma_2+it)|}
 \leq eL(\sigma_2-\sigma_1).
\]
Adding $(\sigma_1-\sigma_2)L$ and exponentiating proves the assertion.
\end{proof}

\subsection{The contour and its contributions}

Throughout this subsection, use the parameters of Theorem~\ref{thm:strip}:
\[
 L=\log x,\qquad S=\sqrt{L\log L},\qquad
 y=e^S,\qquad \log y=S,\qquad u=L/S.
\]
Fix $\epsilon=1/1000$. All sufficiently large thresholds below depend only on $\delta$ and are uniform for $x^{3/4}\leq h\leq x$.

\medskip
\noindent\textbf{Parameter preparation.}
Put $a_0=1-\alpha$. Proposition~\ref{prop:location} gives
\begin{align}\label{eq:parameters}
 a_0
 &=\frac{\log(u\log(u+1))+O(1)}S
 =O\!\left(\frac{\log(u+1)}S\right)=O(1/u),
 \notag\\
 y^{a_0}
 &=\exp\!\left(\log(u\log(u+1))+O(1)\right)
 \asymp u\log(u+1).
\end{align}
Here $a_0>0$ for large $x$, and the first line uses $(\log L)/S=1/u$ and $\log(u+1)=O(\log L)$. Thus, for sufficiently large $x$,
\[
 y\leq x^\epsilon,\qquad a_0\leq\epsilon\delta,\qquad
 \alpha>1-\epsilon,\qquad S\geq(\log L)^2.
\]
Define
\begin{equation}\label{eq:heights}
 D=\frac8{\epsilon L},\qquad
 t_0=Dy^{a_0},\qquad
 H=\frac{x e^{L^{9/10}}}{h},\qquad
 T=H+10\sqrt y.
\end{equation}
The key size estimates are
\begin{align*}
 t_0&\asymp\frac{u\log(u+1)}L
       =\frac{\log(u+1)}S,\\
 e^{L^{9/10}}&\leq H\leq x^{1/4}e^{L^{9/10}}
       =x^{1/4+L^{-1/10}},\\
 2\leq T&\leq x^{1/4+L^{-1/10}}+10x^{S/(2L)}
       \leq x^{3/10}.
\end{align*}
In particular, $H\leq x$. These estimates are uniform throughout the stated range of $h$.

\begin{proof}[Proof of Theorem~\ref{thm:strip}]
\noindent\textbf{Step 1. The right envelope.}
We build a finite contour on which the local shift applies.

Put $b=(1+\epsilon)/2$, and let $\mathcal Z_T$ be the multiset of nontrivial zeros
\[
 \rho=\beta+i\gamma,\qquad \beta\geq\frac12,\qquad |\gamma|\leq T,
\]
counted with multiplicity. By Proposition~\ref{prop:zero-density},
$\#\mathcal Z_T=N(1/2,T)\ll T^{6/5}<\infty$, so this multiset is finite.
For $\rho\in\mathcal Z_T$, write $\nu=1-\beta$; the fixed-strip hypothesis gives $\delta\leq\nu\leq1/2$. Define
\begin{align*}
 \Gamma_\rho^{\mathrm v}
 &=\{1-(1-\epsilon)\nu+it:|t-\gamma|\leq10y^\nu\},\\
 \Gamma_\rho^\pm
 &=\{1-(1-\epsilon)v+i(\gamma\pm10y^v):\nu\leq v\leq1/2\},\\
 \Gamma_\rho
 &=\Gamma_\rho^{\mathrm v}\cup\Gamma_\rho^+\cup\Gamma_\rho^-.
\end{align*}
The curves $\Gamma_\rho^\pm$ join the endpoints of $\Gamma_\rho^{\mathrm v}$ to the line $\Re(s)=b$. No zero with $|\gamma|>T$ can contribute at a height $|t|\leq H$, since every point of its attached path satisfies $|t-\gamma|\leq10\sqrt y$.

Define the central path by
\begin{align*}
 \Gamma_0^{\mathrm v}
 &=\{\alpha+it:|t|\leq t_0\},\\
 \Gamma_0^\pm
 &=\{1-v\pm iDy^v:a_0\leq v\leq(1-\epsilon)/2\},\\
 \Gamma_0
 &=\Gamma_0^{\mathrm v}\cup\Gamma_0^+\cup\Gamma_0^-.
\end{align*}
The curves $\Gamma_0^\pm$ start at $\alpha\pm it_0$ and end on $\Re(s)=b$. Let
\[
 \Gamma_{\mathrm{rem}}=\{b+it:-H\leq t\leq H\}
\]
be the remainder vertical segment. The whole central path lies within $|t|<H$, since
\[
 t_0=Dy^{a_0}\leq Dy^{(1-\epsilon)/2}
 =\exp\!\left(\frac{1-\epsilon}{2}S+O(\log L)\right)
 <e^{L^{9/10}}\leq H.
\]
At each height $-H\leq t\leq H$, define
\[
 \sigma_\Gamma(t)=\max\left\{\sigma\in\mathbb R:
 \sigma+it\in\Gamma_{\mathrm{rem}}\cup\Gamma_0
 \cup\bigcup_{\rho\in\mathcal Z_T}\Gamma_\rho\right\},
 \qquad -H\leq t\leq H,
\]
and set
\[
 \Gamma(t)=\sigma_\Gamma(t)+it,\qquad
 \Gamma=\{\Gamma(t):-H\leq t\leq H\},
\]
oriented upward. Thus $\Gamma$ is the right envelope, and is not itself a closed contour. The finite collection of candidate paths joins $\Gamma_{\mathrm{rem}}$ continuously; their explicit constant and logarithmic graphs make the envelope continuous and piecewise smooth. All zero paths lie strictly to the left of $\alpha$, since
$a_0\leq\epsilon\delta<(1-\epsilon)\nu$. Hence
\[
 \frac{1+\epsilon}{2}\leq\Re(s)\leq\alpha
 \qquad(s\in\Gamma),
\]
and $\Gamma_0^{\mathrm v}\subseteq\Gamma$. To distinguish the main segment from its tails, put
\[
 \Gamma_{\mathrm c}=\{\alpha+it:|t|\leq1/S\},\qquad
 \Gamma_{\mathrm{tail}}=\Gamma_0^{\mathrm v}\setminus\Gamma_{\mathrm c}.
\]
Here $t_0S\asymp\log(u+1)\to\infty$, so $\Gamma_{\mathrm c}\subseteq\Gamma_0^{\mathrm v}$ for large $x$. The part supplied only by the remainder line is
\[
 \Gamma_{\mathrm{rem}}^*
 =\Gamma\setminus\left(\Gamma_0\cup
       \bigcup_{\rho\in\mathcal Z_T}\Gamma_\rho\right)
 \subseteq\Gamma_{\mathrm{rem}}.
\]
Finally, write $\sigma_\pm=\sigma_\Gamma(\pm H)$ and define the horizontal connectors
\[
 C_\pm=\{\sigma\pm iH:\sigma_\pm\leq\sigma\leq\alpha\}.
\]
We orient $C_-$ from $\alpha-iH$ to $\Gamma(-H)$, and $C_+$ from $\Gamma(H)$ to $\alpha+iH$. Integrals over subsets of $\Gamma$ inherit its upward orientation; integrals against $|ds|$ do not depend on orientation. In particular, $\Gamma\cap\Gamma_\rho$ and $\Gamma\cap\Gamma_0^\pm$ denote precisely the portions selected by the envelope.

Let $s=\sigma+it$ lie on or to the right of $\Gamma$, with $\sigma<\alpha$, and put
\[
 \eta=\max\!\left(\delta,\frac{1-\sigma}{1-\epsilon}\right).
\]
Then $\delta\leq\eta\leq1/2$ and
$\sigma\geq1-(1-\epsilon)\eta$. We claim that
\begin{equation}\label{eq:local-box}
 \Re(z)>1-\eta,\qquad
 |\Im(z)-t|<10y^\eta
\end{equation}
contains no zeta zero. If $\eta=\delta$, this is the fixed-strip hypothesis. If $\eta>\delta$ and a zero $\rho=1-\nu+i\gamma$ lay in \eqref{eq:local-box}, then $\nu<\eta$. When
$|t-\gamma|\leq10y^\nu$, its vertical path would have real part
$1-(1-\epsilon)\nu>\sigma$ at height $t$. Otherwise,
$|t-\gamma|=10y^v$ for some $\nu<v<\eta$, and its curved path would have real part
$1-(1-\epsilon)v>\sigma$. Both conclusions contradict the envelope. Zeros with real part below $1/2$ cannot lie in \eqref{eq:local-box}.

We next verify the height condition needed for the local shift. Put
$t_*=Dy^{(1-\epsilon)/2}$, the absolute height of the endpoints of
$\Gamma_0^\pm$. Recall that $\sigma\geq\sigma_\Gamma(t)$ and $\sigma<\alpha$.
If $|t|\leq t_0$, the central vertical segment gives
$\sigma_\Gamma(t)=\alpha$, so no such point exists.
For $t_0<|t|\leq t_*$, the central curve has parameter
\[
 v_t=\frac{\log(|t|/D)}{\log y},\qquad
 a_0<v_t\leq\frac{1-\epsilon}{2},\qquad |t|=Dy^{v_t}.
\]
Its point at height $t$ has real part $1-v_t$. Since the envelope is at
least as far right as this candidate point,
\[
 \sigma\geq\sigma_\Gamma(t)\geq1-v_t,
 \qquad
 Dy^{1-\sigma}\leq Dy^{v_t}=|t|.
\]
For $t_*<|t|\leq H$, we instead use
$\sigma\geq\sigma_\Gamma(t)\geq b=(1+\epsilon)/2$, which gives
\[
 Dy^{1-\sigma}\leq Dy^{1-b}
 =Dy^{(1-\epsilon)/2}=t_*<|t|.
\]
Thus in every applicable case,
\[
 |t|\geq Dy^{1-\sigma}.
\]
Finally, $\sigma<\alpha<1$ implies $y^{1-\sigma}\geq1$.
Using $D=8/(\epsilon L)$, we obtain
\[
 |t|\geq\frac8{\epsilon L}y^{1-\sigma}
 \geq\frac4{\epsilon L}\bigl(y^{1-\sigma}+1\bigr).
\]
Lemma~\ref{lem:local-shift}, applied with $d=\delta$, $e=\epsilon$ and
$\sigma_2=\alpha$, now gives
\begin{align}
 |x^s Z_y(s)|
 &\leq
 x^{-(1-\epsilon)(\alpha-\sigma)}
 |x^{\alpha+it}Z_y(\alpha+it)|
 \notag\\
 &\leq
 W x^{-(1-\epsilon)(\alpha-\sigma)}.
 \label{eq:shift-bound}
\end{align}
For $\sigma=\alpha$, the second bound follows from the triangle inequality for the Euler product. The same reasoning applies on $C_-$ and $C_+$, which lie on or to the right of the envelope at heights $-H$ and $H$.

\medskip
\noindent\textbf{Step 2. Paths associated with zeros.}
We sum the absolute contributions of all retained zero paths.

For $\rho=1-\nu+i\gamma\in\mathcal Z_T$, \eqref{eq:shift-bound} gives
\[
 \int_{\Gamma\cap\Gamma_\rho^{\mathrm v}} |x^sZ_y(s)|\,|ds|
 \leq
 20W x^{(1-\epsilon)a_0}
 \left(\frac{y}{x^{(1-\epsilon)^2}}\right)^\nu.
\]
On either curved path, write $s_\pm(v)=1-(1-\epsilon)v+i(\gamma\pm10y^v)$, where $\nu\leq v\leq1/2$. Then
\[
 |s_\pm'(v)|
 =\sqrt{(1-\epsilon)^2+100S^2y^{2v}}
 \leq1+10Sy^v\ll Sy^v.
\]
Put $A=(1-\epsilon)^2L-S\asymp L>0$. Applying \eqref{eq:shift-bound} on the selected portion and integrating its nonnegative majorant over the full parameter interval, we obtain
\begin{align*}
 \int_{\Gamma\cap\Gamma_\rho^\pm}|x^sZ_y(s)|\,|ds|
 &\ll W x^{(1-\epsilon)a_0}S
       \int_\nu^{1/2}x^{-(1-\epsilon)^2v}y^v\,dv\\
 &\leq W x^{(1-\epsilon)a_0}S
       \int_\nu^\infty e^{-Av}\,dv
 =W x^{(1-\epsilon)a_0}\frac{S}{A}e^{-A\nu}.
\end{align*}
Thus
\[
 \int_{\Gamma\cap\Gamma_\rho^\pm} |x^sZ_y(s)|\,|ds|
 \ll
 W x^{(1-\epsilon)a_0}
 \frac{S}{(1-\epsilon)^2L-S}
 \left(\frac{y}{x^{(1-\epsilon)^2}}\right)^\nu.
\]
Since
$(1-\epsilon)a_0\leq\epsilon\nu$ and $S/L\leq\epsilon$, all three pieces for this zero contribute at most
\[
 \int_{\Gamma\cap\Gamma_\rho}|x^sZ_y(s)|\,|ds|
 \leq CWx^{-(1-4\epsilon)\nu}
 \leq CWx^{-(99/100)\nu}.
\]

We now sum these bounds using zero density. Enumerate the zeros in
$\mathcal Z_T$ as $\rho_j=1-\nu_j+i\gamma_j$, counted with multiplicity,
where $1\leq j\leq M$ and $\delta\leq\nu_j\leq1/2$.
By Proposition~\ref{prop:zero-density} and $T\leq x^{3/10}$,
\[
 \#\{j:\nu_j\leq v\}
 =N(1-v,T)
 \leq CT^{(12/5)v}
 \leq Cx^{(18/25)v}
 \qquad(\delta\leq v\leq1/2).
\]
For $v>1/2$, the same upper bound holds by using the count at $v=1/2$,
since all $\nu_j\leq1/2$.

Put $m_0=\lfloor\delta L\rfloor$. Since $\nu_j\geq\delta$, we have
$\nu_jL\geq\delta L\geq m_0$, so $\lfloor\nu_jL\rfloor\geq m_0$.
Thus we group the indices according to $m\leq\nu_jL<m+1$,
and only integers $m\geq m_0$ need be considered. The preceding estimate gives
\[
 \#\{j:m\leq\nu_jL<m+1\}
 \leq Cx^{(18/25)(m+1)/L}
 =Ce^{(18/25)(m+1)},
\]
while each weight in this group satisfies
$x^{-(99/100)\nu_j}\leq e^{-(99/100)m}$.
Since $99/100-18/25=27/100$, summing the resulting geometric series yields
\begin{align}
 \sum_{\rho\in\mathcal Z_T}x^{-(99/100)\nu}
 &=\sum_{j=1}^M x^{-(99/100)\nu_j}\notag\\
 &\leq Ce^{18/25}\sum_{m\geq m_0}e^{-(27/100)m}
 =\frac{Ce^{18/25}e^{-(27/100)m_0}}{1-e^{-27/100}}
 \ll x^{-27\delta/100}.
 \label{eq:weighted-count}
\end{align}
Here the last inequality uses $m_0>\delta L-1$. The estimate also holds
when $\mathcal Z_T$ is empty.

By Proposition~\ref{prop:mass}, $W\ll\Psi(x,y)L$. Increasing the threshold so that
$L\leq x^{14\delta/100}$, we obtain
\begin{equation}\label{eq:zero-path-bound}
 \sum_{\rho\in\mathcal Z_T}\int_{\Gamma\cap\Gamma_\rho}
 |x^sZ_y(s)|\,|ds|
 \ll_\delta
 \Psi(x,y)x^{-13\delta/100}.
\end{equation}
The shift estimate is used only on $\Gamma\cap\Gamma_\rho$. We then integrate its nonnegative majorant over each full candidate parameter interval. Summing these bounds is valid even if different paths overlap, since any resulting overcount only enlarges the upper bound.

\medskip
\noindent\textbf{Step 3. The remainder line and horizontal connectors.}
We control the portions outside the central and zero paths.

On $\sigma=(1+\epsilon)/2$, the inequalities
$\alpha>1-\epsilon$ and \eqref{eq:shift-bound} imply
\[
 \int_{\Gamma_{\mathrm{rem}}^*}|x^sZ_y(s)|\,|ds|
 \leq
 2HWx^{-1/2+2\epsilon}
 \ll \Psi(x,y)x^{-19/100}.
\]
For the last inequality, use $H\leq x^{3/10}$,
$W\ll\Psi(x,y)L$, and absorb $L$ into the difference between the exponents
$-1/5+2\epsilon$ and $-19/100$.

On each horizontal connector $C_\pm$, integration of \eqref{eq:shift-bound} in $\sigma$ gives
\begin{equation}\label{eq:horizontal-mass}
 \int_{C_\pm}|x^sZ_y(s)|\,|ds|
 \ll\frac WL
 \ll\frac{\Psi(x,y)}{\sqrt u}.
\end{equation}
Indeed,
$\sqrt{LS}/L=1/\sqrt u$. At height $|t|=H$,
\[
 \left|\frac{(x+h)^s-x^s}{s}\right|
 \ll\frac{|x^s|}{H},
\]
because $h\leq x$ and the real parts lie in a fixed bounded interval. Hence both horizontal contributions to the counting integral are
\begin{equation}\label{eq:horizontal-count}
 \left|\frac1{2\pi i}
 \left(\int_{C_-}+\int_{C_+}\right)
 Z_y(s)\frac{(x+h)^s-x^s}{s}\,ds\right|
 \ll\frac{\Psi(x,y)}{H\sqrt u}.
\end{equation}

\medskip
\noindent\textbf{Step 4. The central tails and curves.}
We retain the decay on the central curves in addition to the local shift.

On the vertical segment with $1/S\leq|t|\leq t_0$,
Proposition~\ref{prop:decay} and \eqref{eq:parameters} give
\[
 |Z_y(\alpha+it)|
 \ll Z_y(\alpha)
 \exp\!\left(-\frac{cu}{(\log(u+1))^2}\right).
\]
The segment has length $O(\log(u+1)/S)$, so
\begin{align}\label{eq:vertical-tail}
 \int_{\Gamma_{\mathrm{tail}}}|x^sZ_y(s)|\,|ds|
 &=\int_{1/S<|t|\leq t_0}
 |x^{\alpha+it}Z_y(\alpha+it)|\,dt
 \notag\\
 &\ll
 W\frac{\log(u+1)}S
 \exp\!\left(-\frac{cu}{(\log(u+1))^2}\right)
 =
 o\!\left(\frac{\Psi(x,y)}u\right).
\end{align}
The final comparison uses $W\ll\Psi(x,y)L$, $L/S=u$, and
\[
 u^2\log(u+1)
 \exp\!\left(-\frac{cu}{(\log(u+1))^2}\right)
 \longrightarrow0.
\]

For the selected portions $\Gamma\cap\Gamma_0^\pm$, use the parametrizations
\[
 s_\pm(v)=1-v\pm iDy^v,\qquad
 a_0\leq v\leq\frac{1-\epsilon}{2}.
\]
By \eqref{eq:parameters}, the ratio $t_0/a_0$ is bounded below by a positive constant. Thus
\[
 \frac{t^2}{t^2+a_0^2}\gg1
 \qquad(|t|\geq t_0).
\]
The curve's whole height is less than $\exp(S^{3/2-1/10})$, so
Proposition~\ref{prop:decay}, with $r=1/10$, contributes an additional factor $e^{-cu}$ at $\alpha+it$.

Also,
\[
 |s_\pm'(v)|\leq2DSy^v
\]
for large $x$, since $DSy^{a_0}\asymp\log(u+1)\to\infty$.
To specify where the local shift is used, put
\[
 I=\left[a_0,\frac{1-\epsilon}{2}\right],\qquad
 E_\pm=\{v\in I:s_\pm(v)\in\Gamma\},\qquad
 t_\pm(v)=\pm Dy^v.
\]
For $v\in E_\pm$, the point $s_\pm(v)$ lies on the actual contour, and
$\alpha-\Re s_\pm(v)=v-a_0$. Thus the first inequality in
\eqref{eq:shift-bound} and the preceding decay estimate give
\begin{align*}
 |x^{s_\pm(v)}Z_y(s_\pm(v))|
 &\leq x^{-(1-\epsilon)(v-a_0)}
       |x^{\alpha+it_\pm(v)}Z_y(\alpha+it_\pm(v))|\\
 &\ll We^{-cu}x^{-(1-\epsilon)(v-a_0)},
 \qquad v\in E_\pm.
\end{align*}
Combining this with the arc-length bound, we obtain
\[
 |x^{s_\pm(v)}Z_y(s_\pm(v))|\,|s_\pm'(v)|
 \ll We^{-cu}DSx^{-(1-\epsilon)(v-a_0)}y^v,
 \qquad v\in E_\pm.
\]
The parametrizations are injective, so
\begin{align*}
 \int_{\Gamma\cap\Gamma_0^\pm}|x^sZ_y(s)|\,|ds|
 &=\int_{E_\pm}|x^{s_\pm(v)}Z_y(s_\pm(v))|\,|s_\pm'(v)|\,dv\\
 &\ll We^{-cu}DS\int_{E_\pm}
          x^{-(1-\epsilon)(v-a_0)}y^v\,dv\\
 &\leq We^{-cu}DS\int_{a_0}^{(1-\epsilon)/2}
          x^{-(1-\epsilon)(v-a_0)}y^v\,dv.
\end{align*}
The last step enlarges only the integral of the nonnegative majorant;
no shift estimate is required on $I\setminus E_\pm$.
Substituting $w=v-a_0$ and extending this majorant integral to infinity gives
\begin{align}
 \int_{\Gamma\cap\Gamma_0^\pm}
 |x^sZ_y(s)|\,|ds|
 &\ll
 W e^{-cu}DSy^{a_0}
 \int_0^\infty
 e^{-((1-\epsilon)L-S)w}\,dw
 \notag\\
 &=
 W e^{-cu}
 \frac{DSy^{a_0}}{(1-\epsilon)L-S}
 \notag\\
 &\ll
 W e^{-cu}\frac{\log(u+1)}L
 =
 o\!\left(\frac{\Psi(x,y)}u\right).
 \label{eq:central-curve}
\end{align}
The denominator is positive and comparable with $L$. The same positive majorant bounds either sign of the curve. Local shifting is used only on selected portions of $\Gamma$, while integration over the whole parameter range enlarges the bound.

\medskip
\noindent\textbf{Step 5. The main term and uniform error.}
We assemble the contour estimates without requiring $h/x$ to tend to zero.

Write the counting integrand as
\[
 \mathcal F(s)=Z_y(s)\frac{(x+h)^s-x^s}{s}.
\]
The finite Euler product is holomorphic for $\Re(s)>0$, and $\mathcal F$ has no pole in the region of deformation. With the connector orientations specified in Step~1, Cauchy's theorem gives the exact identity
\[
 \int_{\alpha-iH}^{\alpha+iH}\mathcal F(s)\,ds
 =\int_{C_-}\mathcal F(s)\,ds
  +\int_\Gamma\mathcal F(s)\,ds
  +\int_{C_+}\mathcal F(s)\,ds.
\]
Since $\Gamma_{\mathrm c}\subseteq\Gamma$, we split the middle integral as
\[
 \int_\Gamma\mathcal F(s)\,ds
 =\int_{\Gamma_{\mathrm c}}\mathcal F(s)\,ds
  +\int_{\Gamma\setminus\Gamma_{\mathrm c}}\mathcal F(s)\,ds.
\]

Write $r=h/x$. For $0\leq r\leq1$,
\begin{equation}\label{eq:kernel-identity}
 \frac{(x+h)^s-x^s}{s}
 =
 x^s\int_0^r(1+v)^{s-1}\,dv.
\end{equation}
On the contours under consideration, this gives
\[
 \left|\frac{(x+h)^s-x^s}{s}\right|
 \ll r|x^s|.
\]
On $\Gamma_{\mathrm c}$, one further integration gives
\[
 \frac{(x+h)^s-x^s}{s}
 =
 x^s\bigl(r+O(r^2|s-1|)\bigr).
\]
This estimate remains valid for $r=1$.

Multiplying this expansion by $Z_y(s)$ and integrating separates the
central contribution into a main integral and a weighted error:
\begin{align*}
 \frac1{2\pi i}\int_{\Gamma_{\mathrm c}}\mathcal F(s)\,ds
 &=\frac r{2\pi i}\int_{\Gamma_{\mathrm c}}x^sZ_y(s)\,ds+E_{\mathrm c},\\
 |E_{\mathrm c}|
 &\ll r^2\int_{\Gamma_{\mathrm c}}
       |s-1|\,|x^sZ_y(s)|\,|ds|.
\end{align*}
Since $S=\log y$, the segment $\Gamma_{\mathrm c}$ is exactly the one in
Proposition~\ref{prop:central}. Its main-integral and weighted absolute-integral
estimates give, respectively,
\begin{align*}
 \frac1{2\pi i}\int_{\Gamma_{\mathrm c}}x^sZ_y(s)\,ds
 &=\alpha\Psi(x,y)
   \left(1+O\!\left(\frac1u+\frac Sy\right)\right),\\
 |E_{\mathrm c}|
 &\ll r^2\Psi(x,y)\frac{\log(u+1)}S.
\end{align*}
Consequently,
\[
 \frac1{2\pi i}\int_{\Gamma_{\mathrm c}}\mathcal F(s)\,ds
 =
 \alpha r\Psi(x,y)
 \left(1+O\!\left(\frac1u+\frac Sy\right)\right)
 +
 O\!\left(
 r^2\Psi(x,y)\frac{\log(u+1)}S
 \right).
\]
To compare both errors with the main term $\alpha r\Psi(x,y)$, note that
\[
 \frac{S/y}{1/u}=\frac{uS}{y}=\frac Ly\longrightarrow0,
 \qquad
 \frac{|E_{\mathrm c}|}{\alpha r\Psi(x,y)}
 \ll\frac r\alpha\frac{\log(u+1)}S
 \ll\frac1u.
\]
Here $L/y\to0$ follows from $y=e^S$ and $S/\log L\to\infty$;
the second estimate uses $r\leq1$, $\alpha>1/2$, and
$\log(u+1)/S=O(1/u)$. Hence, uniformly for the stated range of $h$,
\[
 \frac1{2\pi i}\int_{\Gamma_{\mathrm c}}\mathcal F(s)\,ds
 =\alpha r\Psi(x,y)\left(1+O(1/u)\right).
\]

The remaining portion of $\Gamma$ is covered by the paths estimated in Steps~2--4. Thus
\begin{align*}
 \left|\int_{\Gamma\setminus\Gamma_{\mathrm c}}\mathcal F(s)\,ds\right|
 &\leq\int_{\Gamma\setminus\Gamma_{\mathrm c}}|\mathcal F(s)|\,|ds|\\
 &\ll r\Bigg(
  \sum_{\rho\in\mathcal Z_T}\int_{\Gamma\cap\Gamma_\rho}
        |x^sZ_y(s)|\,|ds|
  +\int_{\Gamma_{\mathrm{rem}}^*}|x^sZ_y(s)|\,|ds|\\
 &\hspace{2.5cm}
  +\int_{\Gamma_{\mathrm{tail}}}|x^sZ_y(s)|\,|ds|
  +\sum_{\pm}\int_{\Gamma\cap\Gamma_0^\pm}
        |x^sZ_y(s)|\,|ds|\Bigg)\\
 &=o\!\left(\frac{r\Psi(x,y)}u\right).
\end{align*}
Overlaps among the covering paths cause no problem because all terms on the right are nonnegative. By \eqref{eq:horizontal-count}, the contribution of $C_-\cup C_+$ divided by $r\Psi(x,y)$ is at most
\[
 \frac{Cx}{hH\sqrt u}
 =
 \frac{Ce^{-L^{9/10}}}{\sqrt u}
 =
 o(1/u).
\]
Finally, the error in Proposition~\ref{prop:perron}, divided by
$r\Psi(x,y)$, is at most
\[
 C\frac{x}{\Psi(x,y)}L e^{-L^{9/10}}
 \leq
 CL\exp(3S-L^{9/10})
 =
 o(1/u),
\]
where \eqref{eq:long-density} was used.

Every threshold was chosen uniformly before $h$ ranged over
$[x^{3/4},x]$. We have proved the stated relative estimate. Increase
$X_\delta$ so that its relative error is at most $1/2$. Then
$\alpha>1/2$ and \eqref{eq:long-density} imply
\[
 \Psi(x+h,y)-\Psi(x,y)
 \geq\frac12\alpha\frac hx\Psi(x,y)
 \geq\frac h4e^{-3S}.
\]
\end{proof}

\section{Arithmetic transfer and the final bound}\label{sec:arithmetic}

In this section, we transfer the density conclusion of Theorem~\ref{thm:strip} to actual integer circuits, using the established certificate and marked-slice transfer of \cite{Zhang2026}.

We follow the straight-line program convention of de Melo and Svaiter
\cite{deMeloSvaiter1996} and the notation of
\cite[Definition~1.1]{Zhang2026}.

\begin{defn}[Arithmetic complexity]\label{defn:complexity}
An integer straight-line program begins with $1$ and $2$; each subsequent
gate adds, subtracts or multiplies two available values. Values computed
earlier may be reused. The length is the number of gates, independent of
the bit lengths of their operands. For $n\in\mathbb Z$, let $\tau(n)$
denote the minimum length of such a program producing $n$. For an integer
$k\geq2$, define
\[
 \mu(k)=\min_{m\in\mathbb Z_{>0}}\tau(mk!).
\]
\end{defn}

Only these three binary operations are counted as gates. Other integer
constants must be constructed, and division, modular reduction, gcd,
determinant and group operations must be implemented by the allowed
operations whenever they are used in a circuit.
Starting from $1$ alone changes $\tau$ by at most one operation.
Allowing negative nonzero multipliers in the definition of $\mu$ changes
the minimum by at most a constant, since the sign can be reversed using
a bounded number of arithmetic operations.

\subsection{The cited family and arithmetic inputs}

\begin{defn}[The marked quintic family]\label{defn:family}
For integers $H\geq1$ and
$\boldz=(z_1,\ldots,z_5)\in\{0,\ldots,H-1\}^5$, put
\begin{align}
 t_1&=1+3z_1,\qquad t_i=3z_i\quad(2\leq i\leq5),
 \label{eq:family-parameters}\\
 f_{\boldz}(X)
 &=(1+3^{10}t_5)X^5
   +\sum_{i=1}^4 3^{2i}t_iX^i
   -\sum_{i=1}^5t_i.
 \label{eq:family}
\end{align}
Let $C_{\boldz}$ be the smooth completion of
$Y^2=f_{\boldz}(X)$, and set
\[
 P_{\boldz}=(1/9,1/243),\qquad
 D_{\boldz}=[P_{\boldz}-\infty].
\]
These are the family and marked class of
\cite[(7)--(9)]{Zhang2026}. The same polynomial formulas define the family for $\boldz\in\F_p^5$.
For a prime $p\geq3$, write
\[
 x_p=p^2-\frac12p^{3/2},\qquad
 I_p=[x_p,x_p+p^{3/2}],\qquad
 S_p=\sqrt{\log x_p\log\log x_p}.
\]
\end{defn}

For the permitted integer vectors, smoothness, good reduction at $3$ and infinite order of $D_{\boldz}$ are proved in
\cite[Lemma~4.1]{Zhang2026}. We use the following two consequences as cited results.

\begin{prop}[Charged positive certificates, {\cite[Proposition~5.1]{Zhang2026}}]\label{prop:certificate}
There is an absolute constant $C_0\geq1$ such that, for every pair of integers
$N,H\geq1$ and every permitted vector $\boldz$, there is a positive integer
$Z(N,\boldz)$ with
\[
 \tau(Z(N,\boldz))
 \leq
 C_0\left(
 1+\lceil\log_2N\rceil+\lceil\log_2(H+1)\rceil
 \right).
\]
Its circuit is fixed from the rational data $N,\boldz$, independently of a tested prime. Its output is divisible by $2$ and $3$. For every prime $p\geq5$, its output is divisible by $p$ if either
\begin{enumerate}
\item $f_{\boldz}$ modulo $p$ fails to be a squarefree quintic; or
\item $f_{\boldz}$ modulo $p$ is a squarefree quintic and
$[N](D_{\boldz}\bmod p)=0$ in
$\Jac(C_{\boldz}\bmod p)(\F_p)$.
\end{enumerate}
The bound is for circuits in Definition~\ref{defn:complexity}.
\end{prop}

The charge in Proposition~\ref{prop:certificate} includes all five literal coordinates, scalar scheduling, bounded-degree coefficient arithmetic, and every retained inverse of a rationally nonzero branch test. Complete divisor arithmetic is expanded into scalar instructions in
\cite[Proposition~5.1]{Zhang2026}; rational failures are then recorded by
\cite[Lemma~3.2]{Zhang2026}, and the nonzero integer output is squared. The scalar schedule uses at most $2\lfloor\log_2N\rfloor$ group additions before that expansion. These statements supply no free scalar register, division, gcd, determinant, zero-test or group-operation gate.

The algebraic references behind the cited construction are complete composition, reduction and unique reduced representatives in
\cite[Theorems~10.3.14, 10.3.21 and~10.4.1]{Gal12}, the general reduction algorithm in \cite[Section~2, Algorithm~1]{KS06}, and compatibility of relative Jacobians with good reduction in
\cite[Section~8, Theorem~8.1, and the proof of Corollary~12.3]{Milne21}. Only the general reduction algorithm from \cite{KS06} is relevant; its special endomorphism families are not used.

\begin{prop}[The full marked-slice probability, {\cite[Lemma~6.1]{Zhang2026}}]\label{prop:slice}
Let $p>8100$ be prime and let $\cA\subseteq I_p\cap\Z$. For uniform
$\boldz\in\F_p^5$, with $f_{\boldz}$ given by
\eqref{eq:family-parameters}--\eqref{eq:family},
\[
 \Pr\!\left(
 \begin{array}{c}
 f_{\boldz}\text{ is a squarefree quintic},\\
 \#\Jac(C_{\boldz})(\F_p)\in\cA
 \end{array}
 \right)
 \geq
 \frac{\max(|\cA|-p,0)}
 {48000p^{3/2}(\log p)^2(\log\log p)^2}.
\]
\end{prop}

Proposition~\ref{prop:slice} transfers the full quintic model count of
\cite[Theorem~1.1, p.~106]{LPP02}. That theorem allows arbitrary nonzero leading coefficients and at most $p$ exceptional integer orders. The transfer in \cite[Lemma~6.1]{Zhang2026} uses a bijection of triples to pass to the marked slice. It does not restrict the original count to monic quintics. We use this probability statement directly.

\subsection{A density-to-complexity theorem}

\begin{thm}[Transfer from smooth-order density]\label{thm:density-transfer}
Assume that fixed constants $c_b>0$, $c_d\geq0$ and an integer
$p_\star\geq3$ satisfy the following condition: for every prime
$p\geq p_\star$, the interval $I_p$ contains at least
\[
 p^{3/2}e^{-c_dS_p}
\]
positive integers whose prime divisors are all at most $e^{c_bS_p}$.
Then there is a constant $C>0$, independent of $k$, such that for every integer
$k\geq2$ there is a positive integer $m_k$ with
\[
 \tau(m_k k!)
 \leq
 \exp\!\left(C\sqrt{A_k\log A_k}\right),
 \qquad A_k=\log(16k^2).
\]
More precisely, let $C_0\geq1$ be the constant of
Proposition~\ref{prop:certificate}. Choose an integer $p_0>8100$, with
$p_0\geq p_\star$, such that
\[
 p^{3/2}e^{-c_dS_p}\geq2p
 \qquad\text{for every prime }p\geq p_0.
\]
Then one may take
\[
 D_0=2p_0+384002(6C_0+1)+3,\qquad
 C=c_b+c_d+6+\log D_0.
\]
\end{thm}

\begin{proof}
The proof applies Propositions~\ref{prop:certificate} and~\ref{prop:slice} to one common finite sampling space. We include the sampling losses, common scalar, prime valuations and all gate costs.

\medskip
\noindent\textbf{Step 1. Fixed constants and the sampler.}
We fix every threshold before choosing $k$.

For $p\geq3$,
\[
 1<\log x_p\leq2\log p,\qquad
 S_p\leq\sqrt{2\log p\log(2\log p)}.
\]
Consequently $S_p=o(\log p)$, and
$\tfrac12\log p-c_dS_p\to\infty$. This proves that one integer $p_0$ with the stated properties exists.

Fix $k\geq p_0$, and put
\[
 A_k=\log(16k^2),\qquad S_k=\sqrt{A_k\log A_k}.
\]
Choose
\begin{equation}\label{eq:sampling-parameters}
 B=\lceil e^{c_bS_k}\rceil,\qquad
 H=10k,\qquad
 \delta_k=\frac{e^{-c_dS_k}}{192000A_k^4},\qquad
 J_k=\left\lceil\frac{\log k+1}{\delta_k}\right\rceil.
\end{equation}
Each sample consists of five independent uniform integer coordinates in
$\{0,\ldots,H-1\}$. These formulas specify the sampling parameters; the circuit costs are accounted for below.

\medskip
\noindent\textbf{Step 2. Density on the actual grid.}
We retain the probability bound after deleting every exceptional order.

Let $p_0\leq p\leq k$ be prime, and let $\cT_p$ be the set of $B$-smooth positive integers in $I_p$. Since
$\log x_p\leq2\log k<A_k$, monotonicity of $\sqrt{v\log v}$ for $v>1$ gives
$S_p\leq S_k$. The density hypothesis therefore implies
\[
 |\cT_p|\geq p^{3/2}e^{-c_dS_p}.
\]
The choice of $p_0$ gives
\begin{equation}\label{eq:delete-orders}
 |\cT_p|-p
 \geq\frac12p^{3/2}e^{-c_dS_p}
 \geq\frac12p^{3/2}e^{-c_dS_k}.
\end{equation}
Thus all of the at most $p$ exceptional orders in the model count may be deleted.

Proposition~\ref{prop:slice} now gives, for uniform residues modulo $p$, probability at least
\[
 \frac{e^{-c_dS_k}}
 {96000(\log p)^2(\log\log p)^2}
 \geq
 \frac{e^{-c_dS_k}}{96000A_k^4}
\]
that $f_{\boldz}$ is a squarefree quintic and its Jacobian order belongs to $\cT_p$. The last inequality uses
$0<\log\log p\leq A_k$ and $\log p\leq A_k$.

For one coordinate of the integer sampler, every residue modulo $p$ has mass at least
\[
 \frac{\lfloor H/p\rfloor}{H}
 \geq\frac{1-p/H}{p}
 \geq\frac9{10p}.
\]
Coordinate independence gives every residue vector at least $(9/10)^5$ times its uniform mass. Since
\[
 \left(\frac9{10}\right)^5=\frac{59049}{100000}>\frac12,
\]
the actual integer sampler has the favorable event with probability at least $\delta_k$. The grid and all its parameters are the same for every tested prime.

\medskip
\noindent\textbf{Step 3. One common scalar.}
We make every favorable smooth order divide a single positive integer.

Define
\begin{equation}\label{eq:common-scalar}
 N_k=
 \prod_{\substack{\ell\leq B\\\ell\ {\rm prime}}}
 \ell^{\lfloor A_k/\log\ell\rfloor}.
\end{equation}
The trivial bound on the number of primes gives
\begin{equation}\label{eq:scalar-size}
 \log N_k
 \leq A_k\#\{\ell\leq B\mid\ell\text{ prime}\}
 \leq A_kB.
\end{equation}
Every $B$-smooth positive integer $n\leq16k^2$ divides $N_k$. Indeed, if
$\ell^j\mid n$, then
$j\log\ell\leq\log n\leq A_k$, so
$j\leq\lfloor A_k/\log\ell\rfloor$.

On the favorable event of Step~2, the Jacobian order $n$ is in $I_p$, whose upper endpoint satisfies
\[
 p^2+\frac12p^{3/2}<2p^2\leq2k^2<16k^2.
\]
Hence $n\mid N_k$. The order of an element of a finite group divides the group cardinality, so
\[
 [N_k](D_{\boldz}\bmod p)=0.
\]
Proposition~\ref{prop:certificate} consequently supplies
\[
 p\mid Z(N_k,\boldz).
\]

For each of the finitely many literal vectors in the common grid, fix one certificate circuit from its rational data $N_k,\boldz$. Make this choice once, independently of all tested primes. Its length is at most
\[
 C_0\left(
 1+\lceil\log_2N_k\rceil+\lceil\log_2(H+1)\rceil
 \right).
\]
The scalar's binary digits choose a schedule of at most
$2\lfloor\log_2N_k\rfloor$ group additions, whose scalar expansion is included in this charge by \cite[Proposition~5.1]{Zhang2026}. The five literals and every retained branch-test inverse are included as well. Definition~\eqref{eq:common-scalar} supplies no free prime list, logarithm, exponentiation or register containing $N_k$. Its scheduling cost is bounded through \eqref{eq:scalar-size}.

It follows that, for each tested prime,
\begin{equation}\label{eq:prime-success}
 \Pr_{\boldz}\bigl(p\mid Z(N_k,\boldz)\bigr)\geq\delta_k.
\end{equation}

\medskip
\noindent\textbf{Step 4. Simultaneous prime coverage.}
A finite union bound chooses one successful deterministic list of samples.

Draw $J_k$ independent samples from the same grid. For a fixed tested prime, \eqref{eq:prime-success} bounds the probability that all certificate outputs miss it by
\[
 (1-\delta_k)^{J_k}
 \leq e^{-\delta_kJ_k}
 \leq\frac{e^{-1}}k.
\]
There are fewer than $k$ tested primes. Summing these failure probabilities shows that the probability of missing at least one tested prime is less than $e^{-1}<1$. Thus some deterministic list
\[
 \boldz^{(1)},\ldots,\boldz^{(J_k)}
\]
covers every prime in $[p_0,k]$. This uses independence among samples only.

Construct $p_0!$ directly, using an incrementing counter and a product accumulator, in at most $2p_0$ gates. Concatenate the selected certificate circuits and multiply each positive output into the accumulator. The resulting positive integer
\[
 Q_k=p_0!\prod_{j=1}^{J_k}Z(N_k,\boldz^{(j)})
\]
is divisible by every prime at most $k$. Its circuit length is at most
\begin{equation}\label{eq:coverage-cost}
 2p_0+
 J_k\left[
 C_0\left(
 1+\lceil\log_2N_k\rceil+\lceil\log_2(10k+1)\rceil
 \right)+1
 \right].
\end{equation}
The added $1$ inside the brackets charges each multiplication into the accumulator.

\medskip
\noindent\textbf{Step 5. Factorial valuations.}
We apply the prime-support amplification of
\cite[Lemma~3.1]{Zhang2026} to $Q_k$.

Put $r_k=\lceil\log_2k\rceil$ and append exactly $r_k$ squaring gates. For each prime $p\leq k$,
\[
 v_p(k!)
 =
 \sum_{j\geq1}\left\lfloor\frac{k}{p^j}\right\rfloor
 <
 \frac{k}{p-1}
 \leq k
 \leq2^{r_k}.
\]
Since $p\mid Q_k$,
$v_p(Q_k^{2^{r_k}})\geq2^{r_k}$. For $p>k$, the factorial valuation is zero. Unique factorization gives
\[
 Q_k^{2^{r_k}}=m_k k!
\]
for a positive integer $m_k$. This quotient exists in the proof; it is not computed by a division gate.

\medskip
\noindent\textbf{Step 6. A common explicit constant.}
We bound the full length, including the finite patch and appended squarings.

For $k\geq2$,
\[
 A_k\geq\log64>4,\qquad
 S_k\geq1,\qquad
 \log A_k\leq S_k,\qquad
 \log(10k+1)\leq A_k.
\]
Using \eqref{eq:scalar-size},
\begin{align*}
 1+\lceil\log_2N_k\rceil+\lceil\log_2(10k+1)\rceil
 &\leq3+\frac{A_kB+A_k}{\log2}\\
 &\leq6A_kB.
\end{align*}
For the last inequality, use $B\geq1$, $A_k>4$, and
$3+2A_kB/\log2\leq6A_kB$. Also,
\[
 J_k
 \leq1+192000A_k^5e^{c_dS_k}
 \leq192001A_k^5e^{c_dS_k},
\]
because $\log k+1\leq A_k$ and $A_k^5e^{c_dS_k}\geq1$. Finally,
\[
 B\leq2e^{c_bS_k},
 \qquad
 r_k\leq1+\frac{A_k}{\log2}\leq3A_k.
\]
Adding the squarings to \eqref{eq:coverage-cost}, the full length is at most
\begin{align*}
 &2p_0+
 384002(6C_0+1)A_k^6e^{(c_b+c_d)S_k}
 +3A_k\\
 &\hspace{2cm}\leq
 D_0A_k^6e^{(c_b+c_d)S_k}.
\end{align*}
Taking logarithms and using $\log A_k\leq S_k$ and $S_k\geq1$ gives
\[
 D_0A_k^6e^{(c_b+c_d)S_k}
 \leq
 \exp\!\left((c_b+c_d+6+\log D_0)S_k\right).
\]

For $2\leq k<p_0$, choose $m_k=1$ and construct $k!$ directly in at most
$2k\leq2p_0\leq D_0$ gates. The same bound covers these cases because
$C\geq\log D_0$ and $S_k\geq1$. Thus one constant handles every $k$, all small primes and all finite density exceptions.
\end{proof}

\subsection{Proof of the arithmetic bound}

\begin{proof}[Proof of Theorem~\ref{thm:main}]
Fix $\delta$ as in the hypothesis of Theorem~\ref{thm:main}.
Theorem~\ref{thm:strip} applies with this same $\delta$ and gives a threshold
$X_\delta$, depending only on $\delta$, such that, for every real $x\geq X_\delta$ and every
$x^{3/4}\leq h\leq x$,
\begin{equation}\label{eq:applied-density}
 \Psi(x+h,e^{S_x})-\Psi(x,e^{S_x})
 \geq\frac h4e^{-3S_x},
 \qquad
 S_x=\sqrt{\log x\log\log x}.
\end{equation}
This statement holds at every such center.

For a prime $p\geq5$, take
\[
 x=x_p=p^2-\frac12p^{3/2},
 \qquad h=h_p=p^{3/2}.
\]
Since $x_p\leq p^2$,
$x_p^{3/4}\leq h_p$. Also,
$h_p\leq x_p$ is equivalent to $\sqrt p\geq3/2$, which holds here. Finally,
$x_p\geq p^2/2\to\infty$.

Choose one integer $p_\star\geq5$, depending only on $\delta$, such that, for every prime
$p\geq p_\star$,
\[
 x_p\geq X_\delta,\qquad
 \log x_p>1,\qquad S_p\geq\log4.
\]
Applying \eqref{eq:applied-density}, and using
$(x_p,x_p+h_p]\subseteq I_p$, gives
\begin{align*}
 \#\{n\in I_p\cap\N\mid n\text{ is }e^{S_p}\text{-smooth}\}
 &\geq\frac{h_p}{4}e^{-3S_p}\\
 &\geq h_pe^{-4S_p}.
\end{align*}
The last inequality is exactly $e^{S_p}\geq4$. This is the hypothesis of
Theorem~\ref{thm:density-transfer} with
\[
 c_b=1,\qquad c_d=4.
\]
The parameters $c_b=1$ and $c_d=4$ are absolute, while $p_\star$ may depend
on $\delta$; all are fixed before $k$ varies.

Choose the integer $p_0>8100$ used in that theorem, depending only on
$\delta$, and let $C_0\geq1$ be the absolute certificate constant. The transfer theorem gives an intermediate bound with
\[
 C_1
 =
 11+\log\!\left(2p_0+384002(6C_0+1)+3\right).
\]
Thus, for every $k\geq2$, there is a positive integer $m_k$ with
\[
 \tau(m_k k!)
 \leq
 \exp\!\left(C_1\sqrt{A_k\log A_k}\right),
 \qquad A_k=\log(16k^2).
\]
The existence of $p_0$ follows from $S_p=o(\log p)$, so it does not vary with $k$.
Consequently, $C_1$ depends only on $\delta$.

This composition uses one grid $H=10k$ and one scalar $N_k$, with
\[
 B=\left\lceil\exp\!\left(\sqrt{A_k\log A_k}\right)\right\rceil,
 \qquad \log N_k\leq A_kB.
\]
The proof of Theorem~\ref{thm:density-transfer} charges every certificate, literal construction, product multiplication and scalar-schedule instruction. It also charges the small-prime patch $p_0!$ and the
$\lceil\log_2k\rceil$ squarings supplying factorial valuations. The quotient defining $m_k$ is positive and requires no circuit division.

To obtain the stated bound, put $L=\log k$ and take $k\geq64$. Then
\[
 A_k=2L+\log16\leq3L.
\]
Also $L\geq\log64>3$, so $\log L\geq\log3$ and
\[
 \log A_k\leq\log(3L)\leq2\log L.
\]
Consequently,
\[
 \sqrt{A_k\log A_k}
 \leq\sqrt{6L\log L}
 \leq3\sqrt{L\log L}.
\]
Taking $C=3C_1$ and $k_0=64$ proves Theorem~\ref{thm:main} under the
same fixed-half-plane hypothesis as Theorem~\ref{thm:strip}.
\end{proof}

\section*{Acknowledgements}

The author is grateful to Xiaohua (Michael) Xuan for his support throughout
this work.

\end{document}